\documentclass{amsart}
\usepackage[british]{babel}
\usepackage{amsthm,amssymb}
\usepackage[top=3.6cm,left=3cm,right=3cm,bottom=3.6cm]{geometry}
\usepackage{graphicx}
\usepackage{comment}
\usepackage{mathrsfs}

\usepackage[hidelinks]{hyperref}

\newcommand{\Z}{\mathbb{Z}}
\newcommand{\R}{\mathbb{R}}
\newcommand{\Q}{\mathbb{Q}}
\newcommand{\N}{\mathbb{N}}
\newcommand{\C}{\mathbb{C}}

\renewcommand{\Im}{\operatorname{Im}}
\renewcommand{\Re}{\operatorname{Re}}

\newcommand{\Dom}{\operatorname{Dom}}
\newcommand{\co}[1]{^{\circ {#1}}}
\newcommand{\cent}{\operatorname{Cent}}

 \theoremstyle{plain}
 \newtheorem{thm}{Theorem}[section]
 
 \newtheorem{lem}[thm]{Lemma}

 \theoremstyle {definition}
 
 \theoremstyle{remark}

\title[Analytic maps of parabolic and elliptic type with trivial centralisers]
{Analytic maps of parabolic and elliptic type with trivial centralisers}

\author{Artur Avila}
\email{artur.avila@math.uzh.ch}
\address{Institut f\"{u}r Mathematik, Universit\"{a}t Z\"{u}rich, CH-8057 Z\"{u}rich, Switzerland}

\author{Davoud Cheraghi}
\email{d.cheraghi@imperial.ac.uk}
\address{Department of Mathematics, Imperial College London, London SW7 2AZ, United Kingdom}

\author{Alexander Eliad}
\email{a.eliad18@imperial.ac.uk}
\address{Department of Mathematics, Imperial College London, London SW7 2AZ, United Kingdom}


\date{\today}

\begin{document}

\begin{abstract}
We prove that for a dense set of irrational numbers $\alpha$, the analytic centraliser of the map $e^{2\pi i \alpha} z+ z^2$ near $0$ 
is trivial. 
We also prove that some analytic circle diffeomorphisms in the Arnold family, with irrational rotation numbers, have trivial 
centralisers.
These provide the first examples of such maps with trivial centralisers.  
\end{abstract}

\maketitle

\section{Introduction}
For $\alpha \in \R$, let $\mathcal{H}_\alpha^\omega$ denote the set of germs of holomorphic diffeomorphisms of $(\C,0)$ of the form 
\[h(z)=e^{2\pi i \alpha} z+ O(z^2),\]
defined near $0$. We also consider the class $\mathcal{C}_\alpha^\omega$ of orientation preserving analytic diffeomorphisms of the 
circle $\R/\Z$ with rotation number $\alpha$. 
Let $\mathcal{H}^\omega=\cup_{\alpha \in \R} \mathcal{H}_\alpha^\omega$ and 
$\mathcal{C}^\omega=\cup_{\alpha \in \R} \mathcal{C}_\alpha^\omega$.

The analytic \textit{centraliser} of an element $h \in \mathcal{H}_\alpha^\omega$, denoted by $\cent(h)$, is the set of elements of 
$\mathcal{H}^\omega$ which commute with $h$ near $0$. 
From dynamical point of view, any element of $\cent(h)$ is a conformal symmetry of the dynamics of $h$, that is, 
the conformal change of coordinates $g$ which conjugate $h$ to itself, $g^{-1} \circ h \circ g=h$.  
Evidently, $\cent(h)$ forms a group, where the action is the composition of the elements. 
For every $k \in \Z$, a suitable restriction of the $k$-fold composition $h\co{k}$ is defined near $0$ and belongs to $\cent(h)$.
If the only elements of $\cent(h)$ are of the form $h \co{k}$ for some $k\in \Z$, it is said that $h$ has a \textit{trivial centraliser}.
In the same fashion, for $h\in \mathcal{C}^\omega$, the collection $\cent(h)$ of elements of $\mathcal{C}^\omega$ which 
commute with $h$ enjoys the same features.  

\begin{thm}\label{T:elliptic}
There is a dense set of $\alpha \in \R \setminus \Q$ such that $\cent(e^{2\pi i \alpha} z+ z^2)$ is trivial. 
\end{thm}

The above theorem is proved using a successive perturbation argument and the following statement for parabolic maps which we 
prove in this paper.

\begin{thm}\label{T:parabolic}
For every $p/q \in \Q$, $\cent(e^{2\pi i p/q} z+ z^2)$ is trivial. 
\end{thm}

The main idea we employ to prove the above theorems also allows us to deal with analytic circle diffeomorphisms in the Arnold family, 
\[S_{a,b}(x)=x+ a + b \sin (2\pi x),\]
for $a \in \R$ and $b \in (0,1/(2\pi))$. 

\begin{thm}\label{T:elliptic-circle}
For every $b \in (0,1/(2\pi))$ there is $a \in \mathbb{R}$ such that $\cent(S_{a,b})$ is trivial and the rotation number of $S_{a,b}$ 
belongs to $\mathbb{R}\setminus \mathbb{Q}$.  
\end{thm}

Indeed, we prove that for each fixed $b \in (0, 1/(2\pi))$, the set of rotation numbers of the maps $S_{a,b}$ which have an 
irrational rotation number and a trivial centraliser is dense in $\mathbb{R}$.
The above theorem is obtained from a successive perturbation argument and the analogue of Theorem~\ref{T:parabolic} for 
maps $S_{a,b}$ with a parabolic cycle. 

The main tool used to deal with parabolic maps is Ecalle cylinders and horn maps, first studied and applied by 
Douady-Hubbard \cite{DH84} and Lavaures \cite{La89}.

To our knowledge, Theorems~\ref{T:elliptic} and \ref{T:elliptic-circle} provide the first examples in $\mathcal{H}^\omega$ and 
$\mathcal{C}^\omega$ with irrational rotation numbers and trivial analytic centralisers.
Below we briefly explain how these results fit in the frame of the dynamics of such analytic diffeomorphisms.

When an element $h \in \mathcal{H}_\alpha^\omega$, for $\alpha \in \R \setminus \Q$, is locally conformally conjugate to its 
linear part near $0$, $\cent (h)$ is a large set. 
That is, if $\phi^{-1} \circ h \circ \phi(w)= e^{2\pi i \alpha} w$ near $0$, for some $\phi  \in \mathcal{H}^\omega$,
then for any $\mu \in \C \setminus \{0\}$, $h$ commutes with the map $z \mapsto  \phi(\mu \phi^{-1}(z))$. 
Indeed, here, $\cent(h)$ is isomorphic to $\C \setminus \{0\}$. 
The problem of understanding $\cent(h)$ precedes the problem of local conjugation of $h$ to its linear part.
That is because, the space of solutions for the conjugation problem are the right-cosets of $\cent(h)$.
In this spirit, the size of $\cent(h)$ may be thought of a measure of linearisability of $h$ near $0$.
The same argument applies to analytic circle diffeomorphisms.

For $h \in \mathcal{H}^\omega$, $\cent(h)$ projects onto a subgroup of $\R/\Z$ through $g \mapsto \log g'(0)/(2 \pi i)$. 
Similarly, for $h \in \mathcal{C}^\omega$, one maps $g \in \cent(h)$ to its rotation number. 
Let $\mathcal{G}(h) \subset \R/\Z$ denote the image of this projection. 

By remarkable results of Siegel and Herman \cite{Sie42,Her79} there is a full-measure set $\mathscr{C} \subset \R\setminus \Q$ 
such that for every $\alpha \in \mathscr{C}$, any $h \in \mathcal{H}_\alpha^\omega \cup \mathcal{C}^\omega_\alpha$ is 
analytically linearisable. 
But, for generic choice of $\alpha$, there are $h \in \mathcal{H}_\alpha^\omega$ and $h \in \mathcal{C}^\omega_\alpha$ 
which are not linearisable \cite{Cre38,Ar61}.
We note that if $f$ and $g$ commute, and one of them is linearisable at $0$, then the other one must also be linearisable 
through the same map.
This implies that if $h \in \mathcal{H}_\alpha^\omega \cup \mathcal{C}^\omega_\alpha$ is not linearisable, then 
$\mathcal{G}(h) \subseteq (\R \setminus \mathscr{C})/\Z$.
However, by a profound result of Moser \cite{Moser1990}, $\mathcal{G}(h)$ may not be any subgroup of that set.
That is because there is an arithmetic restriction on the rotations of commuting non-linearisable maps.
The optimal size of $\mathcal{G}(h)$, for nonlinearisable $h$ in $\mathcal{H}_\alpha^\omega$ and $\mathcal{C}_\alpha^\omega$, 
remains open.
This complication is due to the rich structure of the local dynamics of such maps near $0$, see \cite{PerezMarco95,Che17} 
and the references therein. 
However, a complete solution for smooth circle diffeomorphisms is presented in \cite{FayKhan2009}. \nocite{Brj71}

In \cite{Her79,Yoc95,Yoc02}, Herman and Yoccoz carry out a ground breaking study of the centraliser and conjugation problem 
for circle diffeomorphisms and germs of holomorphic diffeomorphisms of $(\C, 0)$. 
In particular, Herman proves the existence of $C^\infty$ circle diffeomorphisms with irrational rotation number 
having uncountably many $C^\infty$ symmetries, and Yoccoz proves the existence of $C^\infty$ circle diffeomorphisms 
with irrational rotation numbers and trivial centralisers. 
Perez-Marco in \cite{PerezMarco95} elaborated a construction of Yoccoz to build elements $h \in \mathcal{H}^\omega$ and 
$h \in \mathcal{C}^\omega$, with irrational rotation number, such that $\mathcal{G}(h)$ is uncountable. 
His construction provides remarkable examples where $\mathcal{G}(h)$ contains infinitely many elements of finite order. 
In this paper we close the problem of the existence of maps in $\mathcal{H}^\omega$ and $\mathcal{C}^\omega$ 
with irrational rotation number and trivial centraliser. 
In light of the above discussions, our result shows that quadratic polynomials and the Arnold family provide the least 
linearisable elements in $\mathcal{H}^\omega$ and $\mathcal{C}^\omega$, respectively. 
This is consistent with the spirit of Yoccoz's argument in \cite{Yoc95}, that is, if some $e^{2\pi i \alpha} z+ z^2$ is linearisable, 
then any element of $\mathcal{H}^\omega_\alpha$ is linearisable.

It is worth noting that the commutation problem for rational functions of the Riemann sphere was already studied by Fatou and 
Julia in 1920's \cite{Jul22,Fat23} using iteration methods. 
A complete classification of such pairs was successfully obtain by Ritt \cite{Ritt23}, using topological and analytic methods, 
and was reproved by Eremenko \cite{Erem89} using modern iteration techniques. 
If iterates of $g$ and $h$ are not identical, modulo conjugation, they are either power maps, Chebyshev polynomials, or Latt\`es maps.
The global commutation problem for entire functions of the complex plane still remains open, 
although substantial progress has been made so far, see for instance \cite{Ganapathy1959,Bak62,Langley99,Ng2001,BRS16}. 
The global commutation problem on higher dimensional complex spaces has been widely studied using iteration methods 
in recent years, see \cite{Dinh-Sibony2002,DinhSibony2004,Kaufmann2018} and the references therein. 
For an extensive discussion on the centraliser and conjugation problems in low-dimensions one may refer to \cite{Kopell70} 
and the more recent survey article \cite{OFarellRogiskaya2010}. 

\nocite{Veselov87} 

\section{parabolic case}
Fix an arbitrary rational number $p/q \in \Q$ with $(p,q)=1$. 
Also fix an arbitrary $g$ in $\cent(Q_{p/q})$. 

The map $F=Q_{p/q}^{\circ q}$ has a parabolic fixed point at $0$ with multiplier $+1$, and there are $q$ attracting directions. 
It follows that the parabolic fixed point of $F$ at $0$ has multiplicity $q+1$. 
That is, the Taylor series expansion of $F$ near $0$ is of the form 
\begin{equation}\label{E:F-defn}
F(z)=Q_{p/q}^{\circ q}(z)=z+\sum_{k=q+1}^{2^q}a_k z^k,
\end{equation}
with $a_{q+1}\neq0$. 

\begin{lem}\label{L:derivative-qth-root}
We have $g'(0)^q=1$.
\end{lem}

\begin{proof}
Let $g(z)=\sum_{k=1}^{\infty} b_k z^k$ denote the Taylor series expansion of $g$ about $0$.
First we show that $b_1 \neq 0$. 
Assume, for a contradiction, that $n\geq 2$ is the smallest positive integer with $b_n \neq 0$. 
Note that $F \circ g= g \circ F$ near $0$. 
By identifying the coefficient of $z^{n+q}$ in the Taylor series expansion of $F \circ g$ and $g \circ F$ we conclude that  
$b_{n+q}+nb_na_{q+1}=b_{n+q}$. Since $a_{q+1}\neq0$, that gives us $b_n=0$, which contradicts the choice of $n$.

Now we identify the coefficients of $z^{q+1}$ in the power series expansions of $F \circ g$ and $g \circ F$, and obtain 
$b_{q+1}+b_1^{q+1}a_{q+1}=b_{q+1}+b_1 a_{q+1}$. This implies that $(b_1^{q+1}-b_1)a_{q+1}=0$. 
Since $a_{q+1}\neq0$ and $b_1\neq 0$, we must have $b_1^q=1$.
\end{proof}

By Lemma~\ref{L:derivative-qth-root}, there is an integer $j$ with $0 \leq j \leq q-1$ such that $(Q_{p/q}\co{j} \circ g)'(0)=1$. 
Consider the holomorphic map 
\begin{equation}\label{E:G-defn}
G(z)=Q_{p/q}\co{j} \circ g, 
\end{equation}
which is defined near $0$ and commutes with $F$. 

\begin{lem}\label{L:multiplicity}
The multiplicity of $G$ at $0$ is $q+1$. That is, $G(z)=z+ \sum_{i=q+1}^\infty b_i z^i$, with $b_{q+1} \neq 0$. 
\end{lem}

\begin{proof}
Assume that $G(z)=z+ b_{n+1}z^{n+1}+ b_{n+2}z^{n+2} + \dots$ is a convergent Taylor series with $b_{n+1}\neq 0$. 
Observe that 
\begin{align*}
F \circ G(z)& 
= \left(z+ b_{n+1}z^{n+1}+ b_{n+2}z^{n+2} + \dots\right ) \\
& \quad + a_{q+1} \left (z+ b_{n+1}z^{n+1}+ b_{n+2}z^{n+2} + \dots \right )^{q+1}   \\  
& \qquad  \vdots \\
& \quad  + a_{q+j} \left (z+ b_{n+1}z^{n+1}+ b_{n+2}z^{n+2} + \dots \right )^{q+j} \\
& \qquad \vdots \\
&= \left(z+ b_{n+1}z^{n+1}+ b_{n+2}z^{n+2} + \dots\right ) \\ 
& \quad + a_{q+1} \left (z^{q+1}+ b_{n+1} (q+1) z^{q+n+1} +\dots\right )  \\
& \qquad \vdots \\
& \quad + a_{q+j} \left(z^{q+j} + b_{n+1} (q+j) z^{q+n+j}  + \dots  \right ) \\
& \qquad \vdots. 
\end{align*}
The coefficient of $z^{q+n+1}$ in the above expansion is 
\[b_{q+n+1}+a_{q+1}b_{n+1}(q+1)+ a_{q+n+1}.\]
Similarly, the coefficient of $z^{n+q+1}$ in the expansion of $G \circ F$ is 
\[a_{q+n+1}+b_{n+1}a_{q+1}(n+1)+b_{q+n+1}.\]
Since $F \circ G = G \circ F$ near $0$, the above values must be identical. 
Using $a_{q+1} \neq 0$ and $b_{n+1} \neq 0$, we conclude that $q=n$. 
\end{proof}

We shall use the theory of Leau-Fatou flower, Fatou coordinates, and horn maps to exploit the local dynamics of $F$ near $0$.  
One may refer to \cite{Mi06} and \cite{Do94} for the basic definitions and constructions we present below, 
although conventions may be different. 

For $s >0$, define the open sets 
\[\Omega_{att}^s = \{\zeta\in \C \mid \Re \zeta >  s - |\Im \zeta|\}, \quad 
\Omega_{rep}^s = \{\zeta\in \C \mid  \Re \zeta < -s+ |\Im \zeta|\}.\]
Also, consider the map $I: \C \setminus \{0\} \to \C \setminus \{0\}$,
\[I(z)= \frac{-1}{ q a_{q+1} z^q}.\]
For $s>0$ there are holomorphic and injective branches of $I^{-1}$ defined on $\Omega_{att}^s$ and $\Omega_{rep}^s$. 

Consider two complex numbers $v_{att}$ and $v_{rep}$ such that 
\[q a_{q+1} v_{att}^q= -1 , \quad v_{rep}= e^{-\pi i/q} v_{att}.\]
Evidently, $I(v_{att})=+1$ and $I(v_{rep})=-1$. 
For $s>0$, there is an injective and holomorphic branch of $I^{-1}$ defined on $\Omega^{s}_{att}$ such that 
$I^{-1}(\Omega^{s}_{att})$ contains $\varepsilon v_{att}$, for sufficiently small $\varepsilon >0$. 
Similarly, there is an injective branch of $I^{-1}$ defined on $\Omega^{s}_{rep}$ such that $I^{-1}(\Omega^{s}_{rep})$
contains $\varepsilon v_{rep}$, for sufficiently small $\varepsilon >0$. 
From now on, we shall fix these choices of inverse branches for $I^{-1}$ on $\Omega^{s}_{att}$ and $\Omega^{s}_{rep}$.
This is independent of $s>0$. 

Let 
\[W_{att}= \left\{z \in \C\setminus \{0\} \mid |\arg (z/v_{att})| \leq \pi/q \right\},\] 
\[W_{rep}= \left\{z \in \C\setminus \{0\} \mid |\arg (z/v_{rep})| \leq \pi/q \right\},\] 
\[W'_{att}= \left\{z \in \C\setminus \{0\} \mid |\arg (z/v_{att})| \leq \pi/q - \pi/(4q)\right \},\]
\[W'_{rep}= \left\{z \in \C\setminus \{0\} \mid |\arg (z/v_{rep})| \leq \pi/q - \pi/(4q)\right \},\]
where $\arg$ denotes a branch of argument with values in $[-\pi, +\pi]$. 

Let $U$ be a Jordan neighbourhood of $0$ such that $G$ is defined on $U$ and both $G$ and $F$ are injective on $U$. 
Since $F'(0)=1$ and $G'(0)=1$, there is $\delta > 0$ such that $B(0, \delta) \subset U$ and 
\begin{equation}\label{E:pre-petals}
\begin{gathered}
F(W_{att}' \cap B(0, \delta)) \subset W_{att},  \quad F(W_{rep}' \cap B(0, \delta)) \subset W_{rep}, \\
G(W_{att}' \cap B(0, \delta)) \subset W_{att},  \quad G(W_{rep}' \cap B(0, \delta)) \subset W_{rep}.
\end{gathered}
\end{equation}
We may choose $r > 0$ such that 
\begin{equation}\label{E:r-selected}
I^{-1} (\Omega_{att}^r) \subset W_{att}' \cap B(0, \delta), \quad 
I^{-1} (\Omega_{rep}^r) \subset W_{rep}' \cap B(0, \delta).
\end{equation}
Now we may lift $F: W'_{att} \cap B(0, \delta) \to W_{att}$ and $F: W'_{rep} \cap B(0, \delta) \to W_{rep}$ via the change 
of coordinate $I(z)=\zeta$ to define injective holomorphic maps 
\[\tilde{F}_{att}: \Omega_{att}^r \to \C, \quad \text{and} \quad 
\tilde{F}_{rep}: \Omega_{rep}^r \to \C.\]
Straightforward calculations show that $\tilde{F}$ is of the form 
\[\tilde{F}_{att}(\zeta)=\zeta + 1 + O(1/|\zeta|^{1/q}), \quad \tilde{F}_{rep}(\zeta)=\zeta + 1 + O(1/|\zeta|^{1/q}),\]
as $|\zeta| \to +\infty$. 
There is $s>0$ such that, 
\[|\tilde{F}_{att}(\zeta) - (\zeta+1) | \leq 1/4, \quad \forall \zeta \in \Omega_{att}^s,\]
\[|\tilde{F}_{rep}(\zeta) - (\zeta+1) | \leq 1/4, \quad \forall \zeta \in \Omega_{rep}^s.\]
There are injective holomorphic maps 
\[\Phi_{att}: \Omega_{att}^s   \to \C, \quad \Phi_{rep}: \Omega_{rep}^s   \to \C,\]
such that
\[\Phi_{att} \circ \tilde{F}_{att} = \Phi_{att} +1, \quad \text{ on } \Omega_{att}^s, \]
\[\Phi_{rep} \circ \tilde{F}_{rep} = \Phi_{rep} +1, \quad \text{ on } \tilde{F}_{rep}^{-1}(\Omega_{rep}^s).\]
It is known that
\begin{equation}\label{E:estimate-Fatou-coordinate-attracting}
\left| \Phi_{att}(\zeta)/\zeta -1 \right | \to 0, \quad \text{ as } \Re \zeta \to +\infty,
\end{equation}
\begin{equation}\label{E:estimate-Fatou-coordinate-repelling}
\left| \Phi_{rep}(\zeta)/\zeta -1 \right | \to 0, \quad \text{ as } \Re \zeta \to -\infty. 
\end{equation}

Let us define 
\[\mathcal{P}_{att}^s= I^{-1}(\Omega_{att}^s), \quad \mathcal{P}_{rep}^s= I^{-1}(\Omega_{rep}^s).\]
Then, the injective holomorphic maps 
\[\phi_{att}=\Phi_{att} \circ I: \mathcal{P}_{att}^s \to \C, \quad \phi_{rep}=\Phi_{rep} \circ I: \mathcal{P}_{rep}^s \to \C,\]
satisfy
\begin{equation}\label{E:fatou-coordinates}
\begin{gathered}
\phi_{att}  \circ F = \phi_{att}+1,  \quad \text{ on } \mathcal{P}_{att}^s, \\
\phi_{rep} \circ F = \phi_{rep}+1, \quad \text{ on } F^{-1}(\mathcal{P}_{rep}^s).
\end{gathered}
\end{equation}
The map $\phi_{att}$ is an \textit{attracting Fatou coordinate} for $F$, and $\phi_{rep}$ is a \textit{repelling Fatou coordinate} 
for $F$. 

Let 
\[\mu = b_{q+1}/a_{q+1}.\]

\begin{lem}\label{L:translations-on-petals}
There is $t \geq 0$ such that 
\begin{itemize}
\item[(i)]  $G(z)= \phi_{att}^{-1} \circ T_\mu \circ \phi_{att}(z)$, for all $z \in \mathcal{P}_{att}^t$, 
\item[(ii)] $G(z)= \phi_{rep}^{-1} \circ T_\mu \circ \phi_{rep}(z)$, for all $z \in \mathcal{P}_{rep}^t$. 
\end{itemize}
\end{lem}

\begin{proof}
By Equations~\eqref{E:pre-petals} and \eqref{E:r-selected}, we may lift $G: W'_{att} \cap B(0, \delta) \to W_{att}$ 
via the change of coordinate $I(z)=\zeta$ to define an injective holomorphic map 
$\tilde{G}_{att}: \Omega_{att}^r \to \C$. We note that $\tilde{G}_{att}$ is of the form 
\[\tilde{G}_{att}(\zeta)= \zeta + \frac{b_{q+1}}{a_{q+1}} + O\left (\frac{1}{|\zeta|^{1/q}}\right ), \text{ as } |\zeta| \to +\infty.\]
In particular, if $|\zeta|$ is large enough, $|\tilde{G}_{att}(\zeta) - (\zeta+ \mu)| \leq 1$. 
This implies that there is $t> s$ such that 
\[\tilde{G}_{att}(\Omega_{att}^t) \subset \Omega_{att}^s.\] 

Let 
\[V= \Phi_{att} (\Omega_{att}^t).\]
Note that since $\tilde{F}_{att}(\Omega_{att}^t) \subset \Omega_{att}^t$, $V+1 \subset V$. 
By Equation~\eqref{E:estimate-Fatou-coordinate-attracting}, if $\Re \zeta$ is large enough, $|\Phi_{att}(\zeta)-\zeta| \leq |\zeta|/3$. 
This implies that 
\[V/\Z= \C/\Z.\]
Consider the injective holomorphic map 
\[\hat{G}_{att}= \Phi_{att} \circ \tilde{G}_{att}  \circ  \Phi_{att}^{-1}: V \to \C.\] 
Since $F$ commutes with $G$ near $0$, $\tilde{F}_{att}$ commutes with $\tilde{G}_{att}$ on the common domain of definition 
$\Omega_{att}^t$. 
Therefore, for $w\in V$, we have 
\begin{align*}
\hat{G}_{att} \circ T_1(w) & =  \Phi_{att} \circ \tilde{G}_{att}  \circ  \Phi_{att}^{-1} \circ T_1(w)\\
&= \Phi_{att} \circ \tilde{G}_{att}  \circ \tilde{F}_{att} \circ \Phi_{att}^{-1}(w)     \\
&= \Phi_{att} \circ \tilde{F}_{att}  \circ \tilde{G}_{att} \circ \Phi_{att}^{-1}(w)     \\
&=T_1 \circ \Phi_{att} \circ \tilde{G}_{att} \circ \Phi_{att}^{-1} (w) = T_1 \circ \hat{G}_{att}(w). 
\end{align*}
Since $V/\Z= \C/\Z$, the above relation implies that $\hat{G}_{att}$ induces a well-defined injective holomorphic map from $\C/\Z$ to 
$\C/\Z$.
Thus, $\hat{G}_{att}$ is a translation on $V/\Z$, and hence, $\hat{G}_{att}$ is a translation on $V$, say $T_\tau$.  
However, since $\Phi_{att}'(\zeta) \to +1$, as $\Re \zeta \to +\infty$, and $\tilde{G}_{att}(\zeta)$ is asymptotically a translation by 
$\mu$ near $+\infty$, we must have $\tau=\mu$. 
That is, $\hat{G}_{att}=T_\mu$. 

For $z\in \mathcal{P}_{att}^t$, we have 
\begin{align*}
\phi_{att}^{-1} \circ T_\mu \circ  \phi_{att}
&= I^{-1} \circ \Phi_{att}^{-1} \circ T_\mu \circ  \Phi_{att} \circ I \\
&= I^{-1} \circ \Phi_{att}^{-1} \circ \hat{G}_{att} \circ  \Phi_{att} \circ I
=I^{-1} \circ \tilde{G}_{att} \circ I= G.
\end{align*}

\bigskip

Part (ii): 
As in the previous part, we may lift $G: W'_{rep} \cap B(0, \delta) \to W_{rep}$ to obtain an injective holomorphic map 
$\tilde{G}_{rep}: \Omega_{rep}^r \to \C$ of the form $\tilde{G}_{rep}= \zeta+ \mu+ o(1)$, as $|\zeta| \to +\infty$. 
Then, one may repeat the argument in part (i) with $\tilde{F}_{rep}$ and $\Phi_{rep}$. 
\end{proof}

Let $B$ denote the set of $z \in \C$ such that $F\co{n}(z) \to 0$, as $n \to +\infty$. 
Evidently, $\mathcal{P}_{att}^s$ is contained in $B$. Let $B_1$ denote the connected component of $B$ which contains 
$\mathcal{P}_{att}^s$. (That is, $B_1$ is the immediate basin of attraction of $0$ in the direction of $v_{att}$.) 
For every $z\in B_1$, there is $k\in \N$ with $F\co{k}(z) \in \mathcal{P}_{att}^s$. 
By the maximum principle, $B_1$ is a simply connected subset of $\C$. 
We may employ the functional relation in Equation~\ref{E:fatou-coordinates}, to extend $\phi_{att}: \mathcal{P}_{att}^s \to \C$ 
to a holomorphic map 
\[\phi_{att}: B_1 \to \C,\]
such that $\phi_{att} \circ F= \phi_{att}+1$ over all of $B_1$. 
 
Consider the trip 
\[\Pi=\{w\in \C \mid  -t-|\mu|-1 < \Re w < -t \} \subset \Omega_{rep}^t.\]
By the estimate in \eqref{E:estimate-Fatou-coordinate-repelling}, if $w\in \Pi$ with $\Im w$ sufficiently large, 
$\Phi_{rep}^{-1}(w) \in \Omega_{att}^s$, and hence $\phi_{rep}(w) \in B_1$. 
On the other hand, for some $w\in \Pi$, $\phi_{rep}(w)$ does not belong to $B_1$. Otherwise, a neighbourhood of $0$ lies 
in $B_1$, which is not possible since $0$ belongs to the Julia set of $F$. 

Let $\Pi'$ denote the connected component of the set $\{w \in \Pi \mid \phi_{rep}^{-1}(w)\in B_1\}$ which contains 
the top end of $\Pi$. 
We may consider the map 
\[h=\phi_{att} \circ \phi_{rep}^{-1}: \Pi' \to \C.\] 
This is a \textit{horn map} of $F$.  
By the functional equations for $\phi_{att}$ and $\phi_{rep}$, we must have $h(\zeta+1)= h(\zeta)+1$, whenever both side of the 
equation are defined. 
Thus, $h$ induces a holomorphic map 
\[H: \Dom H \to \C,\]
on a punctured neighbourhood of $0$ so that $H \circ e^{2\pi i \zeta}= e^{2\pi i h(\zeta)}$. 
By the estimates in \eqref{E:estimate-Fatou-coordinate-attracting} and \eqref{E:estimate-Fatou-coordinate-repelling}, 
$\Im h(\zeta) \to +\infty$, as $\Im \zeta \to +\infty$. 
This implies that $H$ has a removable singularity at $0$. That is $\Dom H$ contains a neighbourhood of $0$. 
\footnote{
The map $H$ is only defined modulo pre-composition and post-composition by linear maps of the form $w \mapsto \lambda w$. 
This is due to the freedom in the choice of $\phi_{att}$ and $\phi_{rep}$ up to post-compositions with translations. 
However, we are not concerned with those choices here.}

\begin{lem}\label{L:unique-critical-value}
The map $H$ has infinitely many critical points, all mapped to the same value. 
\end{lem}

\begin{proof}
Let $c_1$ denoted the unique critical point of $F$ within $B_1$. 
The map $\phi_{att}$ has a simple critical point at $c_1$. It follows from Equation~\eqref{E:fatou-coordinates} that any 
$z\in B_1$ which is mapped to $c_1$ under some iterate of $F$ is a critical point of $\phi_{att}$. 
The closure of the set of such points is equal to the boundary of $B_1$. 

On the other hand, by Equation~\eqref{E:fatou-coordinates}, those critical points are mapped to 
$\phi_{att}(c_1)$, $\phi_{att}(c_1)-1$, $\phi_{att}(c_1)-2$, \dots. 
Since $\phi_{rep}^{-1}$ is conformal on $\Pi' \subset \Omega_{rep}^t$, we conclude that the only critical values of $h$ are 
at $\phi_{att}(c_1), \phi_{att}(c_1)-1, \phi_{att}(c_1)-2, \dots$. 
All those points project to the same value in $\C/\Z$. 
\end{proof}

\begin{lem}\label{L:H-symmetry}
The map $H$ commutes with $\xi \mapsto e^{2\pi i \mu} \xi$ near $0$. 
\end{lem}

\begin{proof}
By Lemma~\ref{L:translations-on-petals}, $G=\phi_{att}^{-1} \circ T_{\mu} \circ \phi_{att}$ on $\mathcal{P}_{att}^t$, and 
$G = \phi_{rep}^{-1} \circ T_{\mu}\circ  \phi_{rep}$ on $\mathcal{P}_{rep}^t$. 
Thus, 
\[\phi_{att}^{-1} \circ T_{\mu} \circ \phi_{att}= \phi_{rep}^{-1} \circ T_{\mu}\circ  \phi_{rep},\]
at any point in $\mathcal{P}_{att}^t \cap \mathcal{P}_{rep}^t$ where both sides of the equation are defined. 
Equivalently, 
\[T_{\mu} \circ \phi_{att} \circ \phi_{rep}^{-1}= \phi_{att} \circ \phi_{rep}^{-1} \circ T_{\mu},\]
whenever both sides of the equation are defined. 
We note that $T_\mu^{-1}(\Pi') \cap \Pi'$ is a non-empty open set, where both sides of the above equation are defined. 
This implies that the horn map $h$ commutes with $T_{\mu}$. 
Hence, $H$ commutes with the map $\xi \mapsto e^{2\pi i \mu} \xi$.
\end{proof}

\begin{lem}\label{L:mu-integer}
We have $\mu \in \Z$. 
\end{lem} 

\begin{proof}
First note that $\Dom H$ is invariant under multiplication by $e^{2\pi i \mu}$. 
That is, on the set $e^{2\pi i \mu} \cdot \Dom H$ we may define $H$ as $\xi \mapsto e^{2\pi i \mu} H (e^{-2\pi i \mu} \xi)$. 
This matches $H$ on $(e^{2\pi i \mu} \cdot \Dom H) \cap \Dom H$. 

Let $c$ denote a critical point of $H$. Differentiating $H (e^{2\pi i \mu} \xi) = e^{2\pi i \mu} H(\xi)$ at $c$, we note that 
$e^{2\pi i \mu} c$ is a critical point of $H$. 
However, $H(e^{2\pi i \mu} c)= e^{2\pi i \mu} H(c)$ is a critical value of $H$. 
By Lemma~\ref{L:unique-critical-value}, we must have $H(c) = e^{2\pi i \mu} H(c)$, which using $H(c) \neq 0$, 
we conclude that $\mu \in \Z$. 
\end{proof}

\begin{proof}[Proof of Theorem~\ref{T:parabolic}]
By Lemma \ref{L:translations-on-petals}, $G = \phi_{att}^{-1} \circ T_\mu \circ \phi_{att}$ on $\mathcal{P}_{att}^t$, 
and by Lemma~\ref{L:mu-integer}, $\mu$ is an integer. 
Thus, on $\mathcal{P}_{att}^t$, 
\[G=\phi_{att}^{-1} \circ T_1\co{\mu} \circ \phi_{att}
= (\phi_{att}^{-1} \circ T_1 \circ \phi_{att}) \circ (\phi_{att}^{-1} \circ T_1 \circ \phi_{att}) \circ \dots 
\circ (\phi_{att}^{-1} \circ T_1 \circ \phi_{att})= F\co{\mu}.\]
As $\mathcal{P}_{att}^t$ is a non-empty open set, we must have $G=F\co{\mu}$ on a neighbourhood of $0$. 

Looking back at definitions~\eqref{E:F-defn} and \eqref{E:G-defn}, we conclude that 
$(Q_{p/q}\co{q})\co{\mu} = Q_{p/q}\co{j} \circ g$, on a neighbourhood of $0$, for some $0 \leq j \leq q-1$. 
Thus, $g=Q_{p/q}\co{(q \mu -j)}$ near $0$. 
\end{proof}

\section{Elliptic case}
Let $g(z)= \sum_{k=1}^\infty g_k z^k \in \cent(Q_\alpha)$. 
It is easy to see that $|g_1|=1$.
Let us say that $g$ is $r$-\textit{good}, if $|g_k| \leq r^{1-k}$ for all $k \geq 1$.
Note that if $g$ is $r$-good, then it is defined and holomorphic on the disk $|z| < r$. 

\begin{lem} \label{2}
For every $p/q \in \mathbb{Q}$ and every $r>0$, $Q_{p/q}\co{k}$ is $r$-good for only finitely many values of $k \in \Z$.
\end{lem}
\begin{proof}
As $Q_{p/q}$ has a parabolic fixed point at $0$, the family of iterates $\{Q_{p/q}\co{k}\}_{k\geq 0}$ and 
$\{Q_{p/q}\co{-k}\}_{k\geq 0}$ have no uniformly convergent subsequence on any neighbourhood of $0$. 
\end{proof}

We let 
\[K(p/q,r) = \left \{k\in \mathbb{Z}\; ; \; Q_{p/q}\co{k} \text{ is $r$-good}\right \}.\]
By the above lemma, $K(p/q,r)$ is a finite set.

\begin{lem} \label{3}
For every $p/q \in \mathbb{Q}$ and every $r>0$, there exists $\delta(p/q,r)>0$ such that for every $p'/q' \in \mathbb{Q}$ with 
$|p'/q'-p/q| \leq \delta(p/q,r)$ we have $K(p'/q',r) \subseteq K(p/q,r)$. 
\end{lem}

\begin{proof}
By the compactness of the set of $r$-good holomorphic maps, there is $N(r)$ such that any $r$-good map has less than 
$N(r)$ critical points is the disk $|z| < r/2$.

As $L$ tends to $+\infty$, the set of the critical points of $Q_{p/q}\co{L}$ increases, and accumulates on $0$.
Let $L \in \N$ be such that $Q_{p/q}\co{L}$ has at least $N(r)$ critical points in the open disk $|z| < r/2$. 
If $p'/q'$ is close enough to $p/q$, then $Q_{p'/q'}\co{L}$ has at least $N(R)$ critical points in the open disk $|z| < r/2$.  
For $l \geq L$, $Q_{p'/q'}\co{l}$ has at least all those critical points, so it is not $r$-good.

Let $M \in \N$ be such that $Q_{p/q}\co{-M}$, and hence $Q_{p/q}\co{-m}$ for any $m \geq M$, does not extend to the open 
disk $|z| <r$. 
Then, the same is true for $p'/q'$ close to $p/q$.

Finally, if $k \notin K(p/q,r)$ and $-M \leq k \leq L$, $Q_{p'/q'}\co{k}$ may not be $r$-good if $p'/q'$ is too close to $p/q$,
because otherwise one could take limits to conclude that $Q_{p/q} \co{k}$ is $r$-good.
\end{proof}

\begin{lem} \label{4}
For every $p/q \in \mathbb{Q}$, every $r>0$, and every $\epsilon>0$, there exists $\kappa(p/q,r,\epsilon)>0$ which satisfies 
the following. 
For every $\alpha \in \R \setminus \Q$ with $|\alpha-p/q| \leq \kappa(p/q,r,\epsilon)$, and every $g(z)=e^{2 \pi i \beta} z+O(z^2)$ 
which commutes with $Q_\alpha$ and is $r$-good, there exists $k \in K(p/q,r)$ such that 
$|\beta-k p/q|<\epsilon \mod \mathbb{Z}$.
\end{lem}

\begin{proof}
If the result does not hold, we may take a sequence $\alpha_n \to p/q$ and $r$-good maps $g_n(z)=e^{2 \pi i \beta_n} z+O(z^2)$ 
which commute with $Q_{\alpha_n}$. 
By the compactness of the set of $r$-good maps, we may choose a convergent subsequence of the $g_n$ converging to a limit $g$
which is $r$-good and commutes with $Q_{p/q}$. 
Then, $g$ will not be of the form $Q_{p/q} \co{k}$ for some $k \in K(p/q,r)$. 
This contradicts Theorem~\ref{T:parabolic} and Lemma~\ref{2}. 
\end{proof}

\begin{lem} \label{5}
For every $\alpha \in \R \setminus \Q$, if a holomorphic germ of the form $g(z)=e^{2 \pi i k \alpha} z+O(z^2)$, for some 
$k \in \Z$, commutes with $Q_\alpha$, then $g=Q_\alpha\co{k}$ near $0$.
\end{lem}

\begin{proof}
By considering $Q_\alpha\co{-k} \circ g$ instead, we may assume that $k=0$. 
Then, by an inductive argument, one may show that the coefficients of the Taylor series expansion of $g$, except the first term, must 
be $0$. That is, $g(z)=z$. 
\end{proof}

\begin{proof}[proof of Theorem~\ref{T:elliptic}]
Start with any rational number $p_1/q_1$. 
We inductively define a strictly increasing sequence of rational numbers $p_n/q_n$, for $n\geq 1$, so that for all $1 \leq l \leq j<n$ 
we have
\begin{equation} \label {6}
|p_n/q_n-p_j/q_j|<\delta(p_j/q_j,1/j),
\end{equation}
\begin{equation} \label {7}
|p_n/q_n-p_j/q_j|<\kappa(p_j/q_j,1/l,1/j),
\end{equation}
\begin{equation} \label {8}
|p_n/q_n-p_j/q_j|<1/q_j^2.
\end{equation}

Let $\alpha=\lim_{n\to \infty} p_n/q_n$.
Since the sequence $p_n/q_n$ is strictly increasing, it follows from Equation~\eqref{8} that $q_n \to \infty$, as $n\to \infty$, 
and $\alpha \in \R \setminus \Q$.

Taking limit as $n\to \infty$ in Equation~\eqref{7}, we note that $|\alpha-p_j/q_j| \leq \kappa(p_j/q_j,1/l,1/j)$, for every 
$1 \leq l \leq j$.

Assume that $g(z)=e^{2 \pi i \beta} z+O(z^2)$ is a germ of a holomorphic map which commutes with $Q_\alpha$. 
There is $l\geq 1$ such that $g$ is $1/l$-good.

By Equation~\eqref{6} and Lemma~\ref{3}, we obtain $K(p_j/q_j,1/l) \subseteq K(p_l/q_l,1/l)$, for $1 \leq l \leq j$.

By Lemma~\ref{4}, for every $j \geq l$, there exists $k \in \Z$ with $k \in K(p_j/q_j,1/l) \subseteq  K(p_l/q_l,1/l)$ such that 
$|\beta-k p_j/q_j|<1/j \mod \Z$. Taking limits of the latter inequality, as $j \to \infty$, we obtain $\beta=k \alpha$, 
for some $k$ in the same range. 
Combining with Lemma~\ref{5}, we conclude that $g=Q_\alpha\co{k}$ near $0$. 
\end{proof}

\section{Circle maps}
We shall employ techniques from complex dynamics to study the analytic symmetries of the maps $S_{a,b}$. 
So we consider the complexified family of maps $S_{a,b}(z)=z+a+b \sin (2\pi z)$, for $z\in \mathbb{C}$, but real values of $a$ 
and $b$.
Using the projection $z \mapsto e^{2\pi i z}$ from $\mathbb{C}$ to $\mathbb{C}^*= \mathbb{C}\setminus \{0\}$, 
$S_{a,b}$ induces the holomorphic map 
\[f_{a,b}(w)=e^{2\pi i a} w e^{\pi b (w-1/w)}\]
from $\mathbb{C}^*$ to $\mathbb{C}^*$. 
Evidently, $f_{a,b}$ preserves the unit circle $\mathbb{T}=\{ w\in \subset \mathbb{C}\; ; \; |z|=1\}$, and for $a\in \mathbb{R}$ 
and $b\in (0, 1/(2\pi))$, $f_{a,b}$ is a diffeomorphism of $\mathbb{T}$. 
Below we always assume that $a\in \mathbb{R}$ and $b\in (0, 1/(2\pi))$. 

\begin{thm}\label{T:parabolic-circle}
Assume that $f_{a,b}$ has a parabolic cycle on $\mathbb{T}$, for some $a$ and $b$. Then, $\cent(f_{a,b})$ is trivial. 
\end{thm}

Let us fix an arbitrary $f_{a,b}$ which has a parabolic cycle on $\mathbb{T}$, say $\{w_i\}_{i=1}^n$, of period $n\geq 1$.  
By relabelling if necessary, we may assume that $f_{a,b}(w_i)=w_{i+1}$, with the subscripts calculated modulo $n$.
Consider the map 
\[F_{a,b} = f_{a,b}\co{n}: \mathbb{C}^* \to \mathbb{C}^*.\] 
Each $w_i$ is a parabolic fixed point of $F_{a,b}$ with multiplier $+1$. 
For $1 \leq i \leq n$, let $U_i \subset \mathbb{C}^*$ denote the immediate basin of attraction of $w_i$ for the iterates of $F_{a,b}$. 
That is, $U_i$ is the union of the connected components of the basin of attraction of $w_i$ which contain $w_i$ on their boundary. 
The following lemma is a special case of a more general result by Geyer \cite[thm 4.4]{Ge01}. 

\begin{lem}\label{L:parabolic-basins-circle}
For every $1\leq i \leq n$, $U_i$ consists of a single connected component, which is invariant under $\tau$, and contains precisely 
two distinct critical points of $F_{a,b}$.  
Moreover, $\cup_{i=1}^n \overline{U_i}= \mathbb{T}$. 
\end{lem}

\begin{proof}
The critical points of $f_{a,b}$ are the solutions of the equation $f_{a,b}'(w)= e^{2\pi i a} e^{\pi b (w-1/w)} (1+ \pi b (w+1/w))=0$. 
Evidently, if $w$ is a solution of this equation, then $\overline{w}$, $1/w$ and $1/\overline{w}$ are all solutions of the equation. 
Thus, $w=\overline{w}$, and hence, the distinct solution of the equation are of the form $c_1$ and $c_2=\tau(c_1)$, 
for some $c_1 \in (-1,0)$. 

Since $F_{a,b}$ is $\tau$-symmetric, it follows that $\tau(U_i)=U_i$, for $1 \leq i \leq n$. 
Moreover, since $F_{a,b}(w_i)=w_i$, every connected component of each $U_i$ is invariant under $F_{a,b}$.
By a classical result of Fatou, see \cite{Mi06}, every connected component of each $U_i$ contains at least one critical point of 
$F_{a,b}$. On the other hand, the critical points of $F_{a,b}$ are the pre-images of the critical points of $f_{a,b}$.
Since $f_{a,b}(U_i)= U_{i+1}$, it follows that there is $j$ with $1 \leq j \leq n$, such that $U_j$ contains the critical points $c_1$ 
and $c_2$. 
Moreover, $c_1$ and $c_2$ are the only critical points of $F_{a,b}$ inside $U_j$. 
Then, the critical values of $f_{a,b}$ belong to $U_{j+1}$, which is distinct from $U_j$.

By the maximum principle, every connected component of each $U_i$ is a simply connected region. 
Since the critical values of $f_{a,b}$ belong to $U_{j+1}$, any other $U_i$ does not contain any critical values of $f_{a,b}$. 
These imply that for $1 \leq l \leq n-1$ there is a conformal branch of $f_{a,b}\co{-l}$ from $U_j$ to $U_{j-l}$.
Therefore, each $U_i$ contains exactly two critical points of $F_{a,b}$. 

Every connected component of each $U_i$ is invariant under $F_{a,b}$ and $\tau$, and contains at least one critical point of 
$F_{a,b}$. 
Therefore, the number of the critical points in $U_i$ is two times the number of the connected components of $U_i$. 
Since $U_i$ contains exactly two critical points of $F_{a,b}$, $U_i$ consists of a single connected component containing both critical 
points. 

Since each $U_i$ has a single connected component, each $w_i$ has a single attraction vector and a single repulsion vectors. 
As $\mathbb{T}$ is invariant, the attraction and repulsion vectors are the tangent vectors to $\mathbb{T}$ at $w_i$.
Fix an arbitrary $i$ and consider an arc of $\mathbb{T}$ cut off by $w_i$ and $w_{i+1}$ which does not contain any other $w_l$.
This arc is invariant under $F_{a,b}$, and the orbit of any point on this arc must converge to either $w_i$ or $w_{i+1}$. 
Otherwise, there will be another fixed point of $F_{a,b}$ on this arc which is distinct from $w_i$ and $w_{i+1}$, and is either 
attracting or parabolic. 
This is a contradiction since such a cycle requires its own critical points distinct from the grand orbit of $c_1$ and $c_2$.
\end{proof}

By relabelling the points $w_i$, and $U_i$ accordingly, we may assume that $U_1$ contains the critical points $c_1$ and $c_2$ 
of $f_{a,b}$.

Since there is only one attracting direction for $F_{a,b}$ at $w_1$, it follows that the multiplicity of the parabolic fixed point at 
$w_1$ is equal to $+2$. 
As in the previous section, there are attracting and repelling Fatou coordinates 
\[\phi_{att}: \mathcal{P}_{att} \to \mathbb{C}, \qquad \phi_{rep}: \mathcal{P}_{rep} \to \mathbb{C},\]
satisfying the functional equations 
\[\phi_{att} \circ F_{a,b} = \phi_{att}+1 , \qquad \phi_{rep} \circ F_{a,b} = \phi_{rep}+1,\]
with $\phi_{att}(\mathcal{P}_{att})= \Omega_{att}^s$ and $\phi_{rep}(\mathcal{P}_{rep})= \Omega_{rep}^s$ for some $s>0$, 
$F_{a,b}\co{j}$ converges to $w_1$ uniformly on compact subsets of $\mathcal{P}_{att}$ as $j\to +\infty$, 
and $F_{a,b}\co{j}$ converges to $w_1$ uniformly on compact subsets of $\mathcal{P}_{rep}$ as $j \to -\infty$. 
The attracting coordinate may be extended to a holomorphic map $\phi_{att}: U_1 \to \mathbb{C}$ using the above functional 
equation. 

The map 
\[h=\phi_{att} \circ \phi_{rep}^{-1}\]
has a maximal domain of definition, which is $\phi_{rep}^{-1}(U_1)+\mathbb{Z}$. 
This induces a holomorphic map $H$ defined on a neighbourhood of $0$, with $H(0)=0$. 

\begin{lem}\label{L:horn-map-circle}
The horn map $H$ has infinitely many critical points, which are mapped to critical values $v_1$ and $v_2$ satisfying 
$\arg v_1= \arg v_2$. 
\end{lem}

\begin{proof}
Any pre-image of $c_1$ and $c_2$ under $F_{a,b}$ within $U_1$ is a critical point of $\phi_{att}$. 
The set of the accumulation points of those pre-images is equal to the boundary of $U_1$ (which is contained in the Julia set of 
$F_{a,b}$). 
By the functional equation for $\phi_{att}$, $\phi_{att}$ maps those critical points into the set 
$\phi_{att}(c_1)+\mathbb{Z}$ or $\phi_{att}(c_2)+\mathbb{Z}$.   
On the other hand, $\phi_{rep}$ is conformal on $\Omega_{rep}^s$.
This implies that the only critical value of $h$ are contained in  $(\phi_{att}(c_1)+\mathbb{Z}) \cup (\phi_{att}(c_2)+\mathbb{Z})$. 

Since $F_{a,b}$ is $\tau$-symmetric, both $\phi_{att}$ and $\phi_{rep}$ are $\tau$-symmetric.
That is, $\phi_{att} \circ \tau = \overline{\phi_{att}}$ and $\phi_{rep} \circ \tau = \overline{\phi_{rep}}$.
This is due the uniqueness of a Fatou-coordinate up to translation by a constant. 
Combining with the above paragraph, we conclude that  $\overline{\phi_{att}(c_1)}= \phi_{att}(c_2)$, and hence 
the critical values of $H$ have the same argument.
\end{proof}

\begin{proof}[Proof of Theorem~\ref{T:parabolic-circle}]
The proof already starts at the beginning of this section. Fix an arbitrary $f_{a,b}$ with a parabolic cycle $\{w_i\}_{i=1}^n$ of period 
$n$. Let us also fix an arbitrary $g \in \cent(f_{a,b})$.   
The commutation implies that $g(w_1)$ is a periodic point of period $n$ for $f_{a,b}$. 
By Lemma \ref{L:parabolic-basins-circle}, $f_{a,b}$ has a unique periodic cycle, which is $\{w_i\}_{i=1}^n$. 
Therefore, there is an integer $k \geq 1$ such that $f_{a,b} \co{k} \circ g(w_1)=w_1$. 
Let us define the analytic map 
\[G=f_{a,b} \co{k} \circ g: \mathbb{T} \to \mathbb{T}.\]
As $F_{a,b}$ commutes with $G$, $F_{a,b}(w_1)=w_1$, $F_{a,b}'(w_1)=1$ we may repeat Lemma\ref{L:derivative-qth-root}
to conclude that $G'(w_1)=1$. 
On the other hand, since the multiplicity of $F_{a,b}$ at $w_1$ is equal to $+2$, 
we may repeat Lemma \ref{L:multiplicity} to conclude that the multiplicity of $G$ at $w_1$ is also equal to $+2$. 
That is, $G$ is of the form 
\[G(w)= G(w_1)+ (w-w_1)+ b_2 (w-w_1)^2 + \dots,\]
near $0$, with $b_2\neq 0$. 
As in the previous section, we must have $G=\phi_{att}^{-1} \circ T_\mu \circ \phi_{att}$ on $\mathcal{P}_{att}$ and 
$G=\phi_{rep}^{-1} \circ T_\mu \circ \phi_{rep}$ on $\mathcal{P}_{rep}$, where $\mu =2 b_2/F_{a,b}''(0)$.  
Repeating Lemma~\ref{L:H-symmetry}, we conclude that $H$ must commute with the rotation $\xi \mapsto e^{2\pi i \mu} \xi$ 
near $0$.
Now, as in the proof of Lemma~\ref{L:mu-integer}, we use Lemma~\ref{L:horn-map-circle} instead of Lemma~\ref{L:H-symmetry}, 
to say that if $c$ is a critical point of $H$, then we must have $\arg H(c)= \arg (e^{2\pi i \mu} H(c))$. 
This implies that $\Re \mu \in \mathbb{Z}$. 
On the other hand, if $\Im \mu \neq 0$, since the domain of definition of $H$ is invariant under $\xi \mapsto e^{2\pi i \mu}$, 
we conclude that $H$ is defined over all of $\mathbb{C}$. But this is a contraction since $H$ has infinitely many critical points in 
a bounded region of the plane. 
Therefore, $\mu \in \mathbb{Z}$, and hence $G= F_{a,b}\co{\mu}$.
This completes the proof of Theorem~\ref{T:parabolic-circle}
\end{proof}

Fix an arbitrary $b \in (0, 1/(2\pi))$. 
By a general theorem of Poincar\'e, $f_{a,b}$ has a period point on $\mathbb{T}$ if and only if its rotation number 
$\rho(f_{a,b}) \in \mathbb{Q}$. 
Moreover, by classical results, the map $a \mapsto \rho(f_{a,b})$ is an increasing function of $a \in (0, 1)$. 
It is locally strictly increasing at irrational values, that is, if $\rho(f_{a,b}) \in \mathbb{R} \setminus \mathbb{Q}$ for some $a$, 
then for $a'\in (0,1)$ with $a'>a$, $\rho(f_{a',b})> \rho(f_{a,b})$. 
However, at rational values, the map is constant on a closed interval. 
\footnote{The set of $a$ and $b$ where $\rho(f_{a,b})$ is a rational number has non-empty interior, and is known as 
Arnold tongues. One may refer to \cite{Ar61}, \cite{Fag1999} for basic features of those loci, and the global dynamics of the 
complexified standard family.}  

Given $r>1$, we say that an analytic homeomorphism $g: \mathbb{T} \to \mathbb{T}$ is $r$-good, 
if $g$ is holomorphic on the annulus $1/r < |z|< r $ and maps that annulus to the annulus $1/2 < |z| < 2$. 
Evidently, every analytic homeomorphism of $T$ is $r$-good for some $r>1$. 
Moreover, by Schwarz-Pick lemma, for every $r>1$, the class of $r$-good analytic homeomorphisms of $\mathbb{T}$ forms a 
compact class of maps.  

Let us consider the sets 
\[P=\{(a,b) \in (0, 1) \times (0, 1/(2\pi)) \; ; \; f_{a,b} \text{ has a parabolic cycle on } \mathbb{T}\}.\]
and for each $b\in (0, 1/(2\pi))$, 
\[P_b=\{a \in (0, 1) \; ; \; (a,b) \in P\}.\]

\begin{lem}\label{2'}
For every $(a,b) \in P$, $f_{a,b}\co{k}$ is $r$-good for only finitely many values of $k$. 
\end{lem}

\begin{proof}
The proof of Lemma~\ref{2} may be repeated here to show this statement.
\end{proof}

For $(a,b) \in P$, we define 
\[K'(a,b,r) = \left \{k \in \mathbb{Z}\; ; \;    f_{a,b}\co{k} \text{ is $r$-good}\right \}.\]

\begin{lem}\label{3'} 
For every $(a,b) \in P$ and every $r>0$, there exists $\delta'(a,b,r)>0$ such that for every $a' \in P_b$ with 
$|a'-a| \leq \delta'(a,b,r)$ we have 
\[K'(a',b,r) \subseteq K'(a,b,r).\] 
\end{lem}

\begin{proof}
This is the same as the proof of Lemma~\ref{3}. 
\end{proof}

\begin{lem}\label{4'}
For every $(a,b) \in P$, every $r>0$, and every $\epsilon>0$, there exists $\kappa'(a,b,r, \epsilon)>0$ which 
satisfies the following. 
For every $a' \in P_b$ with $|a'-a| \leq \kappa'(a,b,r,\epsilon)$ and $\rho(f_{a', b}) \in \mathbb{R} \setminus \mathbb{Q}$, 
and every $r$-good map $g$ which commutes with $f_{a',b}$, there exists $k \in K'(a,b,r)$ such that 
$|\rho(g)- k \rho(f_{a,b})|<\epsilon \mod \mathbb{Z}$.
\end{lem}

\begin{proof}
The proof is identical to the one for Lemma~\ref{4}. Here one uses the continuity of the map $x \mapsto \rho(f_{x,b})$, 
for $x\in \mathbb{R}$. 
\end{proof}

\begin{lem}\label{5'} 
Assume that $\rho(f_{a,b}) \in \R \setminus \Q$. 
If $g: \mathbb{T} \to \mathbb{T}$ is an analytic map which commutes with $f_{a,b}$ and $\rho(g)= k \rho(f)$ for some $k \in \Z$, 
then $g= f \co{k}$ on $\mathbb{T}$. 
\end{lem}

\begin{proof}
By considering $f_{a,b}\co{-k} \circ g$ instead, we may assume that $\rho(g)=0$. 
By Poincar\'e's theorem, $g$ has a fixed point, and then by the commutation of $f_{a,b}$ and $g$, any iterate of that fixed point by 
$f_{a,b}$ must be a fixed point of $g$. 
Since the orbit of any point in $\mathbb{T}$ by $f_{a,b}$ is dense on $\mathbb{T}$, $g$ has a dense set of 
fixed points. Thus, $g$ is the identity map on $\mathbb{T}$. 
\end{proof}

\begin{proof}[Proof of Theorem~\ref{T:elliptic-circle}]
The proof is similar to the one for Theorem~\ref{T:elliptic}, using Theorem~\ref{T:parabolic-circle} instead of 
Theorem~\ref{T:parabolic}.

Fix an arbitrary $b\in (0, 1/(2\pi))$, and start with an arbitrary $a \in P_b$.
We inductively define an strictly increasing sequence of parameters $a_n \in P_b$, for $n\geq 1$, so that for all $1 \leq l \leq j<n$ 
we have
\begin{equation} \label{6'}
|a_n - a_j |<\delta' (a_j, b,1/j),
\end{equation}
\begin{equation} \label{7'}
|a_n - a_j|< \kappa'(a_j, b,1/l,1/j),
\end{equation}
\begin{equation} \label{8'}
|\rho(f_{a_n,b}) - \rho(f_{a_j,b}) |< 1/q_j^2,
\end{equation}
where $p_j/q_j =\rho(f_{a_j,b}) \in \Q$ and $(p_j, q_j)=1$. 

Let $a= \lim_{n\to \infty} a_n$.
Since the sequence $a_n$ is strictly increasing, the sequence $p_n/q_n$ must be increasing with at most two consecutive terms 
identical. It follows from Equation~\eqref{8'} that $q_n \to \infty$, as $n\to \infty$, and $\rho(f_{a,b}) \in \R \setminus \Q$.

Taking limit as $n\to \infty$ in Equation~\eqref{7'}, we note that $|a - a_j | \leq \kappa'(a_j, b,1/l,1/j)$, for every 
$1 \leq l \leq j$.

Assume that $g$ is a orientation preserving analytic homeomorphism of $\mathbb{T}$ which commutes with $f_{a,b}$. 
There is $l \geq 1$ such that $g$ is $1/l$-good.

By Equation~\eqref{6} and Lemma~\ref{3}, we obtain $K'(a_j, b,1/l) \subseteq K'(a_l,b,1/l)$, for $1 \leq l \leq j$.

By Lemma~\ref{4}, for every $j \geq l$, there exists $k \in \Z$ with $k \in K'(a_j, b,1/l) \subseteq K'(a_l,b,1/l)$ such that 
$|\rho(g) - k p_j/q_j|<1/j \mod \Z$. Taking limits of the latter inequality, as $j \to \infty$, we obtain $\rho(g) = k \rho(f_{a,b})$, 
for some $k$ in the same range. 
Combining with Lemma~\ref{5}, we conclude that $g=f_{a,b} \co{k}$ on $\mathbb{T}$. 
\end{proof}

\textbf{Acknowledgement:} A.~Avila and D.~Cheraghi gratefully acknowledge funding from EPSRC (UK); grant EP/M01746X/1 
--rigidity and small divisors in holomorphic dynamics, while carrying out this research. 

\bibliographystyle{amsalpha}
\bibliography{/Users/davoud/Desktop/19-pc-topology/Data}
\end{document}